\documentclass[12pt]{amsart}
 \usepackage{amssymb,amsmath,amsfonts,latexsym,setspace}
 \usepackage{bm}
 \usepackage{array,graphics,color}
 \theoremstyle{plain}
 \newtheorem{propn}{Proposition}[section]
 \newtheorem{thm}[propn]{Theorem}
 \newtheorem{lemma}[propn]{Lemma}
 
 \newtheorem*{thm*}{Theorem}
 \theoremstyle{definition}

 \theoremstyle{remark}
 \newtheorem{rem}{Remark}
 
\begin{document}

 \title{Trace Formula For Two Variables}

 \author[Chattopadhyay] {Arup Chattopadhyay}
 \address{
 (A. Chattopadhyay) Indian Statistical Institute \\ Statistics and
 Mathematics Unit \\ 8th Mile, Mysore Road \\ Bangalore \\ 560059 \\
 India.}

 \email{arup@isibang.ac.in, 2003arupchattopadhyay@gmail.com}

\author[Sinha]{Kalyan B. Sinha}

\address{
 (K. B. Sinha) J.N.Centre for Advanced Scientific Research\\ and Indian Institute of Science,\\ 
 Bangalore\\ India.}

 \email{kbs@jncasr.ac.in}


\begin{abstract}
A natural generalization of Krein's theorem to a pair of commuting tuples 
$\left(H_1^0,H_2^0\right)$ and $\left(H_1,H_2\right)$ of bounded self-adjoint
operators in a separable Hilbert space $\mathcal{H}$ with
$H_j-H_j^0 = V_j\in \mathcal{B}_2(\mathcal{H})$(set of
all Hilbert-Schmidt operators on $\mathcal{H}$) for $j=1,2,$ 
leads to a Stokes-like formula under trace. A major ingredient in the proof is the finite-dimensional 
approximation result for commuting self-adjoint n-tuples of operators,
a generalization of Weyl-von Neumann-Berg's theorem.
\end{abstract}

 \maketitle
\section{Introduction}\label{sec: intro}

In the following, $\mathcal{H}$ will denote the separable Hilbert space we work in;
$\mathcal{B}(\mathcal{H})$, $\mathcal{B}_1(\mathcal{H})$,
$\mathcal{B}_2(\mathcal{H})$, $\mathcal{B}_p(\mathcal{H})$ the set of bounded, trace class, 
Hilbert-Schmidt class and Schatten-p class operators in $\mathcal{H}$ respectively with  $\|.\|$, $\|.\|_1$, 
$\|.\|_2$, $\|.\|_p$ as the associated norms. Let $H$ be a self-adjoint operator in $\mathcal{H}$ 
with $\sigma(H)$ as 
the spectra and $E_{H}(\lambda)$ as the spectral family and let $\textup{Tr} A$ 
be the trace 
of a trace class operator $A$. Furthermore, $C(X)$ will be the set of all continuous functions 
on a compact
set $X$ and $L^p(Y,dx)$ ($1\leq p\leq \infty$) be the standard Lebesgue spaces, where $Y=[a,b]^2$ or $[a,b]^4$
for the real interval $[a,b]$ and $dx$ is the relevant Lebesgue measure.

Let $H$ and $H_0$ be two possibly unbounded self-adjoint operators in a separable 
Hilbert space 
$\mathcal{H}$ such that $V=H-H_0 \in \mathcal{B}_1(\mathcal{H})$. 
Then Krein proved that 
there exists a unique real-valued $L^1(\mathbb{R})$- function $\xi$ with support in the interval 
$[a,b]$ such that 
\begin{equation}\label{eq: kreinintro}
 \textup{Tr}\left[\phi\left(H\right)-\phi\left(H_0\right)\right] = \int_{a}^{b} \phi^{\prime}
 (\lambda)\xi(\lambda)d\lambda,
\end{equation}
for a large class of functions $\phi$ (where $a=min\{\inf \sigma(H), \inf \sigma(H_0)\}$ 
and $b=max\{\sup \sigma(H),$ $ \sup \sigma(H_0)\}$). The function 
$\xi$ is known as 
Krein's spectral shift function \color{black} and the relation~\eqref{eq: kreinintro}~
is called Krein's trace formula. The original proof of Krein \cite{krein} uses analytic function theory.
But in 1985, Voiculescu 
approached the trace formula ~\eqref{eq: kreinintro}~ from a different direction. If $H$ and 
$H_0$ are bounded, then Voiculescu \cite{Voiculescu} proved that
\begin{equation}
\textup{Tr}\left[p\left(H\right)-p\left(H_0\right)\right] = \lim_{n\longrightarrow \infty} 
\textup{Tr}\left[p\left(H_n\right)-p\left(H_{0,n}\right)\right]\color{black},
\end{equation}
where $p$ is a polynomial and $H_n$, $H_{0,n}$ are finite-dimensional approximation of $H$ and $H_0$
respectively (constructed by adapting Weyl-von Neumann theorem). Then one constructs the spectral shift 
function in the finite dimensional case and finally the
formula is extended to the infinite dimensional case. Later Sinha
and Mohapatra (\cite{sinhamoha},\cite{sinmoha}) used a similar method to get the same result for the unbounded 
self-adjoint case. 
If on the other hand  $H-H_0 = V\in \mathcal{B}_2(\mathcal{H})$, the difference 
$\phi(H) - \phi(H_0)$
is no longer of trace-class and one has to consider instead
$\phi(H) - \phi(H_0) - D^{(1)}\phi(H_0)(V),$
where $D^{(1)}\phi(H_0)(V)$ denotes the Fr$\acute{e}$chet derivative of $\phi$ at $H_0$ acting on
$V$ (see \cite{bhatia}) and 
find a trace formula for the above expression. 
Under the above hypothesis, Koplienko's formula \cite{kop} asserts that there exists a unique function 
$\eta \in L^1(\mathbb{R})$ such that
\begin{equation}\label{eq: intkopeq}
 \textup{Tr}\{\phi(H) - \phi(H_0) - D^{(1)}\phi(H_0)(V)\} = \int_{-\infty}^{\infty} \phi''(\lambda)
 \eta(\lambda)d\lambda
\end{equation}
for rational functions $\phi$ with poles off $\mathbb{R}$.
Gesztesy, Pushnitski and Simon \cite{gespusbir} gave an alternative proof of the formula \eqref{eq: intkopeq} for 
the bounded case and Dykema and Skripka \cite{dyskip} and earlier Boyadzhiev \cite{boya}
obtained the formula \eqref{eq: intkopeq} in the semi-finite von Neumann algebra setting. The present
authors used the finite-dimensional approximation idea to prove the Koplienko formula \cite{chattosinha1}
as well as the third-order trace formula for both bounded and unbounded cases \cite{chattosinha2}. More recently, 
Potapov, Skripka and Sukochev \cite{potaskipsuko} has proven 
the trace-formula for all orders, obtaining a kind of Taylor's theorem
under trace. In fact they have shown in \cite{potaskipsuko},
the existence of $\eta_n\in L^1(\mathbb{R})$ for $n\in \mathbb{N}$
such that 
$$\textup{Tr}\left(\phi(H_0+V)-\sum_{k=0}^{n-1}\frac{1}{k!}
D^{(k)}\phi(H_0)(\underbrace{V,V,\ldots,V}\limits_{k-\text{times}})\right)
=\int_{\mathbb{R}}\phi^{(n)}(\lambda)\eta_n(\lambda)d\lambda,$$
for every sufficiently smooth function $\phi,$ where $H_0$ is a self-adjoint operator 
defined on a Hilbert space $\mathcal{H},$ $V$ is a self-adjoint operator such that
$V\in \mathcal{B}_n(\mathcal{H}),$ $D^{(k)}\phi(H_0)(\underbrace{V,V,\ldots,V}\limits_{k-\text{times}})$
denotes the $k$-th order Fr$\acute{e}$chet derivative of $\phi$ at $H_0$ acting on
$(\underbrace{V,V,\ldots,V}\limits_{k-\text{times}})$ (see \cite{bhatia}) and 
$\phi^{(n)}$ denotes the $n$-th order derivative of the function $\phi$.
It is natural to ask similar questions for a pair of commuting self-adjoint
n-tuples, particularly an appropriate adaptation of Krein's formula \eqref{eq: kreinintro} to
two and higher dimensions. Here our aim is to formulate a relevant question for a pair of
commuting bounded self-adjoint tuples and use the idea of finite dimensional approximation 
to obtain Stokes-like formula under trace. In this context, it should be mentioned that recently
Skripka \cite{skripka} has studied a related problem for commuting contractions.

The Section 2 addresses the issue of finite-dimensional approximation for n-tuples of commuting
bounded self-adjoint operators by adapting Berg's \cite{davidson} extension of Weyl-von Neumann theorem.
Section 3 deals with the spectral integrals of operator functions and reducing the problem into a finite
dimensional case. Finally in Section 4 we have established Stokes-like formula for a class of operator functions
under trace.

\section{Approximation Results}\label{sec: second}
The main result in this section (Theorem \ref{th: WBG}) is an adaptation from the proof 
of Weyl-von Neumann-Berg theorem \cite{davidson} for proving a finite-dimensional approximation in suitable 
Schatten-von Neumann $\mathcal{B}_p$-ideal norm for commuting n-tuples
of bounded self-adjoint operators.
First we need a known simple lemma, the proof of which is given for
the sake of completeness.

\begin{lemma}\label{lmma1}
Let $A\in \mathcal{B}(\mathcal{H})$ be such that $0\leq A\leq I$. Now consider the spectral projections
$E_k = E_{A}\left(\bigcup \limits_{j=1}^{2^{k-1}}\left(2^{-k}(2j-1),2^{-k}(2j)\right]\right)$ for $k\geq 1$.
Then
\begin{equation}\label{berg1}
A = \sum_{k=1}^{\infty}2^{-k}E_k,
\end{equation}
where the right hand side of \eqref{berg1} converges in operator norm.
\end{lemma}
\begin{proof}
We want to show that
\begin{align*}
A 
&= \sum_{k=1}^{\infty}2^{-k}E_k = \sum_{k=1}^{\infty}2^{-k}E_A\left(\bigcup_{j=1}^{2^{k-1}}
\left(2^{-k}(2j-1),2^{-k}(2j)\right]\right)\\
&= \sum_{k=1}^{\infty}2^{-k} \sum_{j=1}^{2^{k-1}} E_A\left(2^{-k}(2j-1),2^{-k}(2j)\right],
\end{align*}
since $\left(2^{-k}(2i-1),2^{-k}(2i)\right] \bigcap \left(2^{-k}(2j-1),2^{-k}(2j)\right] = \emptyset$
for $i \neq j$ and $1\leq i,j \leq 2^{k-1}$. Let
\begin{equation}\label{ind1}
S_N \equiv \sum_{k=1}^{N}2^{-k}E_k = \sum_{k=1}^{N}2^{-k} \sum_{j=1}^{2^{k-1}} E_A
\left(2^{-k}(2j-1),2^{-k}(2j)\right].
\end{equation}
Next  by applying principle of mathematical induction on $N,$ we prove that
\begin{equation}\label{ind2}
S_N = \sum_{m=1}^{2^N-1} m2^{-N}E_A\left(2^{-N}m,2^{-N}(m+1)\right].
\end{equation}
For $N=1,$ $S_1 = 2^{-1} E_A\left(2^{-1}1,2^{-1}(2)\right]
= 1.2^{-1}E_A\left(2^{-1}1,2^{-1}(1+1)\right]$ and therefore the equation \eqref{ind2} is true for $N=1$

Next we assume that the equation \eqref{ind2} is true for $N=l$ i.e.
\begin{equation}\label{ind3}
S_l = \sum_{k=1}^{l}2^{-k} \sum_{j=1}^{2^{k-1}} E_A
\left(2^{-k}(2j-1),2^{-k}(2j)\right] = \sum_{m=1}^{2^l-1} m2^{-l}E_A\left(2^{-l}m,2^{-l}(m+1)\right].
\end{equation}
Therefore by using equation \eqref{ind3} and changing summation index appropriately we get
\begin{equation}\label{chap4sum1}
\begin{split}
S_{l+1} =  \sum_{k=1}^{l+1}2^{-k} \sum_{j=1}^{2^{k-1}} E_A\left(2^{-k}(2j-1),2^{-k}(2j)\right]\\
& \hspace{-7.2cm} = \sum_{k=1}^{l}2^{-k} \sum_{j=1}^{2^{k-1}} E_A\left(2^{-k}(2j-1),2^{-k}(2j)\right]
+ 2^{-(l+1)} \sum_{j=1}^{2^{l}} E_A\left(2^{-(l+1)}(2j-1),2^{-(l+1)}(2j)\right]\\ 
& \hspace{-7.2cm} = \sum_{m=1}^{2^l-1} m2^{-l}E_A\left(2^{-l}m,2^{-l}(m+1)\right] 
+ 2^{-(l+1)} \sum_{m=1}^{2^{l}} E_A\left(2^{-(l+1)}(2m-1),2^{-(l+1)}(2m)\right]\\
& \hspace{-7.2cm} = \sum_{m=1}^{2^l-1} m2^{-l}E_A\left(2^{-l}m,2^{-l}(m+1)\right] 
+ 2^{-(l+1)} E_A\left(2^{-(l+1)},2^{-l}\right] \\
& \hspace{-4cm} +  \sum_{m=2}^{2^{l}} 2^{-(l+1)}E_A\left(2^{-(l+1)}(2m-1),2^{-(l+1)}(2m)\right]\\
& \hspace{-7.2cm} = \sum_{m=1}^{2^l-1} m2^{-l}E_A\left(2^{-l}m,2^{-l}(m+1)\right] 
+ 2^{-(l+1)} E_A\left(2^{-(l+1)},2^{-l}\right] \\
& \hspace{-4cm} +  \sum_{m=1}^{2^{l}-1} 2^{-(l+1)}E_A\left(2^{-(l+1)}(2m+1),2^{-(l+1)}(2m+2)\right].
\end{split}
\end{equation}
But on the other hand the first summation in the equation \eqref{chap4sum1} gives us
\begin{equation}\label{chap4sum2}
\begin{split}
\sum_{m=1}^{2^l-1} m2^{-l}E_A\left(2^{-l}m,2^{-l}(m+1)\right] 
= \sum_{m=1}^{2^l-1} 2m 2^{-(l+1)}E_A\left(2^{-(l+1)}2m,2^{-(l+1)}(2m+2)\right]\\
& \hspace{-14.8cm} = \sum_{m=1}^{2^l-1} 2m2^{-(l+1)}E_A\{\left(2^{-(l+1)}2m,2^{-(l+1)}(2m+1)\right]\cup 
\left(2^{-(l+1)}(2m+1),2^{-(l+1)}(2m+2)\right]\}\\
& \hspace{-14.8cm} = \sum_{m=1}^{2^l-1} 2m 2^{-(l+1)}E_A\left(2^{-(l+1)}2m,2^{-(l+1)}(2m+1)\right]\\
& \hspace{-10cm} + \sum_{m=1}^{2^l-1} 2m 2^{-(l+1)}E_A\left(2^{-(l+1)}(2m+1),2^{-(l+1)}(2m+2)\right].\\
\end{split}
\end{equation}
Combining \eqref{chap4sum1}
and \eqref{chap4sum2}, we conclude that 
\begin{equation*}
\begin{split}
 S_{l+1} = 2^{-(l+1)} E_A\left(2^{-(l+1)},2^{-l}\right] +
 \sum_{m=1}^{2^l-1} 2m 2^{-(l+1)}E_A\left(2^{-(l+1)}2m,2^{-(l+1)}(2m+1)\right] 
 \\
& \hspace {-10cm} + \sum_{m=1}^{2^l-1} (2m+1)2^{-(l+1)}E_A\left(2^{-(l+1)}(2m+1),2^{-(l+1)}(2m+2)\right]\\
&  \hspace {-13.2cm} = \sum_{m=1}^{2^{(l+1)}-1} m 2^{-(l+1)}E_A\left(2^{-(l+1)}m,2^{-(l+1)}(m+1)\right].
\end{split}
\end{equation*}
Therefore the equation \eqref{ind2} is true for $N=l+1$, completing the induction. 
Thus for $f\in \mathcal{H}$,
\begin{align*}
 \left(\sum_{k=1}^{\infty} 2^{-k}E_k\right)f
&= \lim_{N\longrightarrow \infty}\left(\sum_{k=1}^{N} 2^{-k}E_k\right)f\\
&\hspace{-2.5cm} = \lim_{N\longrightarrow \infty}\left(\sum_{m=1}^{2^N-1} m.2^{-N}E_A\left(m.2^{-N},(m+1)
.2^{-N}\right]\right)f
= \int \lambda E_A(d\lambda)f,
\end{align*}
by using the definition of spectral integral of $A$ (see \cite{ajs}). Thus
\begin{equation*}
A = \int \lambda E_A\left(d\lambda\right) = \sum_{k=1}^{\infty} 2^{-k}E_k.
\end{equation*}
\end{proof}

A result due to Weyl and von Neumann \cite{kato} proves that for a self-adjoint operator $A$ that given $\epsilon >0$, 
$\exists K \in \mathcal{B}_2(\mathcal{H})$
such that $\|K\|_2 < \epsilon$ and $A+K$ has pure point spectrum. Later Berg extended this to an 
n-tuples of bounded commuting self-adjoint operators
$\left(A_1,A_2,\ldots,A_n\right)$,
which says that given $\epsilon >0$,~ $\exists$ $\{K_j\}_{j=1}^n$ of compact operators such that 
$\|K_j\|<\epsilon$
$\forall j$ and $\{A_j-K_j\}_{j=1}^n$ is a commuting family of bounded self-adjoint operators with 
pure point spectra.
We extend in the next theorem the ideas of the proof of Berg's result as given in \cite{davidson}.
It is worth mentioning that Voiculescu \cite{Voiculescu2} 
had earlier obtained related (though not the same) results.

\begin{thm}\label{th: WBG}
Let $\{A_i\}_{1\leq i \leq n}$ be a commuting family of bounded self-adjoint operators in an 
infinite-dimensional separable Hilbert space $\mathcal{H}$.
Then there exists a sequence $\{P_N\}$ of finite-rank projections such that $\{P_N\} \uparrow I$ as 
$N\longrightarrow \infty$ and such that there exists a
commuting family of bounded self-adjoint operators $\{B_i^{(N)}\}_{1\leq i\leq n}$ with the properties that
for $p\geq n$ and for each $i$ ($1\leq i\leq n$), as $N\longrightarrow \infty$,
\vspace{0.2in}

(i) $P_N B_i^{(N)} P_N = B_i^{(N)} P_N,$ ~~~(ii) $\left\|A_i-B_i^{(N)}\right\|_p \longrightarrow 0$,
~~~(iii) $\left\|[A_i,P_N]\right\|_p \longrightarrow 0$, 
\vspace{0.2in}

(iv) $\left\|P_NA_iP_N - B_i^{(N)}P_N\right\|_p \longrightarrow 0$ and (v) $\{B_i^{(N)}\} \uparrow A_i$.

\end{thm}
\begin{proof}
 One can assume without loss of generality that  $0\leq A_i \leq I$ for all $1\leq i\leq n$, 
 and therefore for each $i$, 
 by  Lemma \ref{lmma1},
 \begin{center}
 $A_i = \sum \limits_{k=1}^{\infty}2^{-k}E_k^{(i)}$,
\end{center}
where
$E_k^{(i)} = E_{A_i}\left(\bigcup \limits_{j=1}^{2^{k-1}}(2^{-k}(2j-1), 2^{-k}(2j)]\right)$ 
with $E_{A_i}$ the spectral measure associated 
to the bounded self-adjoint operator $A_i.$ 
Next set for $N\in \mathbb{N}$ (the set of natural numbers), 
\begin{equation*}
 \mathcal{L}_N \equiv \text{span} \{ \left[\prod_{k=1}^N \prod_{i=1}^n \left(E_k^{(i)}\right)
 ^{\epsilon}\right]
 f_j ~| ~~1\leq j\leq N;~\epsilon = \pm 1\},
\end{equation*}
where $\{f_1,f_2,....., f_{_N},......\}$ be a countable orthonormal basis of $\mathcal{H}$ and 
$\left(E_k^{(i)}\right)^1 = E_k^{(i)}$ and 
$\left(E_k^{(i)}\right)^{-1} = I- E_k^{(i)}.$
Thus $\mathcal{L}_N$ is a finite dimensional subspace of $\mathcal{H}$ and it has the following properties:

$(a)\hspace{0.1cm}\mathcal{L}_N \subseteq \mathcal{L}_{N+1}$,
$\hspace{0.1cm}(b)
\hspace{0.1cm}\overline{\left(\bigcup \limits_{N=1}^{\infty}\mathcal{L}_N\right)} = \mathcal{H}$
and $(c)\hspace{0.1cm} \textup{dim}\left( \mathcal{L}_N \right) \leq N\left(2^{n}-1\right)^N + N.$
\vspace{0.1in}

{\bf{$(a)$:}}
Now $\left[\prod \limits_{k=1}^N~\prod \limits_{i=1}^n \left(E_k^{(i)}\right)^
{\epsilon}\right] f_j \in \mathcal{L}_N$ ~~for ~~$1\leq j\leq N$, 
\vspace{0.1in}

i.e. $\left(E_N^{(n)}\right)^{\epsilon}\left[\prod \limits_{k=1}^{N-1}~\prod \limits_{i=1}^{n-1} 
\left(E_k^{(i)}\right)^{\epsilon}\right] f_j \in \mathcal{L}_N$ ~~for ~~$1\leq j\leq N$,
\vspace{0.1in}

i.e. $\left(I-E_N^{(n)}\right)\left[\prod \limits_{k=1}^{N-1}~\prod \limits_{i=1}^{n-1} 
\left(E_k^{(i)}\right)^{\epsilon}\right] f_j \in \mathcal{L}_N$ ~~and hence
\vspace{0.1in}

\hspace{0.5cm} $\left(E_N^{(n)}\right)\left[\prod \limits_{k=1}^{N-1}\prod \limits_{i=1}^{n-1} 
\left(E_k^{(i)}\right)^{\epsilon}\right] f_j \in \mathcal{L}_N$ ~~for ~~$1\leq j\leq N$,
\vspace{0.1in}
 
\hspace{-0.6cm} which implies that $\left[\prod \limits_{k=1}^{N-1}\prod \limits_{i=1}^{n-1} 
\left(E_k^{(i)}\right)^{\epsilon}\right] f_j \in \mathcal{L}_N$ for ~$1\leq j\leq N$.
By repeating the above argument we conclude that $\left(E_1^{(1)}\right)^{\epsilon}f_j 
\in \mathcal{L}_N$ for ~$1\leq j\leq N,$
i.e. $E_1^{(1)}f_j,\left(I-E_1^{(1)}\right)f_j \in \mathcal{L}_N$ for ~~$1\leq j\leq N$ and hence
$\{ f_1,f_2,f_3,\ldots,f_N\} \subset \mathcal{L}_N,$ proving that $\overline{\left(\bigcup 
\limits_{N=1}^{\infty}\mathcal{L}_N\right)} = \mathcal{H}$.
\vspace{0.1in}

{\bf{$(b)$:}} Again $\left[\prod \limits_{k=1}^{N+1}~\prod \limits_{i=1}^n 
\left(E_k^{(i)}\right)^{\epsilon}\right]f_j \in \mathcal{L}_{N+1}$ ~~for 
~~$1\leq j\leq N+1$. But on the other hand
\vspace{0.1in}

$\left[\prod \limits_{k=1}^{N+1}~\prod \limits_{i=1}^n \left(E_k^{(i)}\right)^
{\epsilon}\right]f_j = \prod \limits_{i=1}^n\left(E_{N+1}^{(i)}\right)^{\epsilon}
\left[\prod \limits_{k=1}^{N}~\prod \limits_{i=1}^n \left(E_k^{(i)}\right)^{\epsilon}\right]f_j$
\vspace{0.1in}

$= \left(E_{N+1}^{(n)}\right)^{\epsilon}\left[\prod\limits_{i=1}^{n-1}\left(E_{N+1}^{(i)}\right)
^{\epsilon}\right]\left[\prod \limits_{k=1}^{N}~\prod \limits_{i=1}^n \left(E_k^{(i)}\right)
^{\epsilon}\right]f_j$ and hence 
\vspace{0.1in}

$\left(E_{N+1}^{(n)}\right)\left[\prod \limits_{i=1}^{n-1}\left(E_{N+1}^{(i)}\right)^{\epsilon}\right]
\left[\prod \limits_{k=1}^{N}~\prod \limits_{i=1}^n \left(E_k^{(i)}\right)^{\epsilon}\right]f_j,$  
\vspace{0.1in}

\hspace{1cm} $\left(I-E_{N+1}^{(n)}\right)\left[\prod \limits_{i=1}^{n-1}\left(E_{N+1}^{(i)}\right)
^{\epsilon}\right]\left[\prod \limits_{k=1}^{N}~\prod \limits_{i=1}^n \left(E_k^{(i)}\right)
^{\epsilon}\right]f_j \in \mathcal{L}_{N+1}$ ~~for ~~$1\leq j\leq N+1$.
\vspace{0.1in}

Therefore, $\prod \limits_{i=1}^{n-1}\left(E_{N+1}^{(i)}\right)^{\epsilon}
\left[\prod \limits_{k=1}^{N}~\prod \limits_{i=1}^n \left(E_k^{(i)}\right)^{\epsilon}\right]
f_j \in \mathcal{L}_{N+1}$ ~~for ~~$1\leq j\leq N+1$.
\vspace{0.1in}

By repeating the above argument we conclude that 
\vspace{0.1in}

$\left[\prod \limits_{k=1}^{N}~\prod \limits_{i=1}^n \left(E_k^{(i)}\right)^{\epsilon}\right]
f_j \in \mathcal{L}_{N+1}$ for $1\leq j\leq N+1$  
and hence $\mathcal{L}_N \subseteq \mathcal{L}_{N+1}$ 
\vspace{0.1in}

{\bf{$(c)$:}} According to the definition of $E_K^{(i)}$, it follows that for each 
fixed $k\in \mathbb{N}$,
\begin{equation}\label{identity}
 \sum_{\epsilon = \pm 1} \prod_{i=1}^{n}\left(E_k^{(i)}\right)^{\epsilon}=I.
\end{equation}
We claim that for any fixed vector $f\in \mathcal{H}$, the  
$\text{span}~\{\left[\prod\limits_{k=1}^N\prod\limits_{i=1}^n \left(E_k^{(i)}\right)^{\epsilon}\right] f 
: \epsilon = \pm 1\}$
contains at most $\left(2^n-1\right)^N$ linearly independent vectors, without counting $f$. 
We prove our claim by induction on $N$. For $N=1,$ because of the identity \eqref{identity} we conclude that
the $\text{span}~\{\left[\prod\limits_{i=1}^n \left(E_1^{(i)}\right)^{\epsilon}\right] f 
: \epsilon = \pm 1\}$ contains at most $\left(2^n-1\right)$ linearly independently vectors besides $f$. 
Since $\{A_i\}_{1\leq i\leq n}$ is a commuting family,  we have the following:
$$\text{span}~\{\left[\prod\limits_{k=1}^{N+1}\prod\limits_{i=1}^n \left(E_k^{(i)}\right)^{\epsilon}\right] f 
: \epsilon = \pm 1\} : = \text{span}~\{\prod\limits_{i=1}^n \left(E_{N+1}^{(i)}\right)^{\epsilon}
\left[\prod\limits_{k=1}^N\prod\limits_{i=1}^n 
\left(E_k^{(i)}\right)^{\epsilon}\right] f 
: \epsilon = \pm 1\} $$
and thus by the induction hypothesis,
$\text{span}~\{\left[\prod\limits_{k=1}^N\prod\limits_{i=1}^n \left(E_k^{(i)}\right)^{\epsilon}\right] f 
: \epsilon = \pm 1\}$ contains at most $\left(2^n-1\right)^N$ linearly independent vectors, other than $f$.
Therefore using the equation \eqref{identity} we conclude that the 
$\text{span}~\{\left[\prod\limits_{k=1}^{N+1}\prod\limits_{i=1}^n \left(E_k^{(i)}\right)^{\epsilon}\right] f 
: \epsilon = \pm 1\}$ contains at most 
$$2^n\left(2^n-1\right)^N - \left(2^n-1\right)^N = \left(2^n-1\right)^{N+1}$$
number of linearly independent
vectors, other than $f$ itself, completing the induction.
Hence for any fixed vector $f\in \mathcal{H}$, the  
$\text{span}~\{\left[\prod\limits_{k=1}^N\prod\limits_{i=1}^n \left(E_k^{(i)}\right)^{\epsilon}\right] f 
: \epsilon = \pm 1\}$
contains the maximum of possible $\{\left(2^n-1\right)^N +1\}$ linearly independent vectors, including
the vector $f$. This implies that
$\mathcal{L}_N$ contains at most $N\{\left(2^n-1\right)^N +1\}$ number of linearly independent vectors 
and therefore
$\textup{dim} \left(\mathcal{L}_N\right) \leq N\left(2^{n}-1\right)^N + N.$ Now we set
$P_N$ to be the finite rank projection associated with the finite dimensional
subspace $\mathcal{L}_N.$
Then by $(a)$ and $(b)$ the sequence  $\{P_N\}$ increases to $I$. Next we define 
\begin{equation*}
 B_i^{(N)} = \sum _{k=1}^N 2^{-k} E_k^{(i)} + \sum_{k=N+1}^{\infty} 2^{-k} E_k^{(i)} (I-P_k),
\end{equation*}
and observe that since $\{E_k^{(i)}\}_{_{1\leq k\leq N ; 1\leq i\leq n}}$ 
is a commuting family and since each member of that family for fixed $k$ 
commutes with $P_l$ for $1\leq k\leq l,$ it is easy to verify that 
\begin{equation*}
\begin{split}
  & E_k^{(i)} (I-P_k) E_{k'}^{(i)} (I-P_{k'}) = (I-P_k)(I-P_{k'})E_{k'}^{(i)}E_k^{(i)}\\
  &= (I-P_{k'})E_{k'}^{(i)}E_k^{(i)} = E_{k'}^{(i)} (I-P_{k'})E_k^{(i)} (I-P_k),
\end{split}
\end{equation*}
where we have assumed without loss of generality that $k\leq k'$. Thus  $\{B_i^{(N)}\}_{1\leq i\leq n}$ 
is a commuting family of positive self-adjoint
contractions and since $(I-P_k)P_N = 0$ for $k\geq N+1,$ it follows that 
\begin{equation}\label{dfnb}
 P_NB_i^{(N)}P_N = B_i^{(N)}P_N = \sum_{k=1}^N 2^{-k}E_k^{(i)} P_N,
\end{equation}
and hence $B_i^{(N)}P_N $, a finite dimensional self-adjoint operator in the Hilbert space $P_N\mathcal{H}$.
Furthermore, $A_i - B_i^{(N)} = \sum \limits_{k=N+1}^{\infty} 2^{-k} E_k^{(i)} P_k$ and 
\begin{equation*}
\begin{split}
\hspace{-0.8cm} \left\|A_i - B_i^{(N)} \right\|_n \leq \sum_{k=N+1}^{\infty} 2^{-k}\left\|P_k\right\|_n \leq  
\sum_{k=N+1}^{\infty} 2^{-k} \left[k\{1+\left(2^n-1\right)^k\}\right]^{\frac{1}{n}} \\
& \hspace{-9cm}=\sum_{k=N+1}^{\infty}k^{\frac{1}{n}}\left[2^{-nk} + \left(1-2^{-n}\right)^k\right]^{\frac{1}{n}}\\
& \hspace{-9cm}\leq \sum_{k=N+1}^{\infty}k^{\frac{1}{n}} 2^{-k} + \sum_{k=N+1}^{\infty}
k^{\frac{1}{n}} \left[\left(1-2^{-n}\right)^{\frac{1}{n}}\right]^k, 
\end{split}
\end{equation*}
where we have used that for $a,b>0$ ,~$(a+b)^{\frac{1}{n}}\leq (a^{\frac{1}{n}}+b^{\frac{1}{n}})$.
Since for fixed n, $(1-2^{-n})^{\frac{1}{n}}<1$, and since 
$ \sum\limits_{k=1}^{\infty}k^{\frac{1}{n}}\alpha^{k}<\infty$ for $\alpha < 1$,
it follows that for each $i$($1\leq i\leq n$), 
$\left\|A_i - B_i^{(N)} \right\|_n\longrightarrow 0$
as $N\longrightarrow \infty$.
Therefore for any $p\geq n$ we get
\begin{equation*}
 \begin{split}
 \hspace{-2cm} \left\|A_i - B_i^{(N)} \right\|_p\leq \left(\|A_i\|+\|B_i^{(N)}\|\right)^{\left(1-\frac{n}{p}\right)}
 \left\|A_i - B_i^{(N)} \right\|_n^{\frac{n}{p}}\\
 & \hspace{-6.7cm} \leq 2^{\left(1-\frac{n}{p}\right)}
 \left\|A_i - B_i^{(N)} \right\|_n^{\frac{n}{p}}
 \longrightarrow 0
\end{split}
\end{equation*}
as $N\longrightarrow \infty$. Thus
\begin{equation}
 \left\|[A_i, P_N]\right\|_p =
 \left\|\left[A_i-B_i^{(N)}, P_N\right]\right\|_p\leq 2 \left\|A_i-B_i^{(N)}\right\|_p \longrightarrow 0 ~\text{as}
 ~N\longrightarrow \infty,
\end{equation}
for any $p\geq n$. Moreover,
\begin{equation}\label{eqfour}
 \hspace{-0.5cm} \left\|P_NA_iP_N - P_NB_i^{(N)}P_N\right\|_p = \left\|P_N\left(A_i-B_i^{(N)}\right)
 P_N\right\|_p \leq \left\|A_i -B_i^{(N)}\right\|_p \longrightarrow 0
\end{equation}
as $N\longrightarrow \infty$ for any $p\geq n$.
Now from \eqref{dfnb} we have
\begin{equation*}
\begin{split}
 & B_i^{(N+1)} = \sum _{k=1}^{N+1} 2^{-k} E_k^{(i)} P_{N+1}\\
 & \hspace{1.3cm} = \sum _{k=1}^{N} 2^{-k} E_k^{(i)} \left(P_{N+1}-P_N\right)
 + \sum _{k=1}^{N} 2^{-k} E_k^{(i)} P_{N} + 2^{-(N+1)}E_{N+1}^{(i)}P_{N+1}.
\end{split}
\end{equation*}
Thus 
\begin{equation*}
\begin{split}
 B_i^{(N+1)} - B_i^{(N)} 
 = \sum _{k=1}^{N} 2^{-k} \left(P_{N+1}-P_N\right) E_k^{(i)} \left(P_{N+1}-P_N\right)
 + 2^{-(N+1)}P_{N+1} E_{N+1}^{(i)}P_{N+1} \geq 0,
\end{split}
\end{equation*}
since $\{P_N\} \uparrow I$. Finally by using \eqref{eqfour} and the fact that
$P_NA_iP_N \longrightarrow A_i$ strongly as $N\longrightarrow \infty$, we conclude that 
$B_i^{(N)}P_N \longrightarrow A_i$ strongly as $N\longrightarrow \infty$. This completes the proof.
\end{proof}

\begin{rem}
The choice that $0\leq A_i\leq I$ does not materially affect the calculations of 
Theorem \ref{th: WBG}. For if $C_i \in \mathcal{B}(\mathcal{H})$ ($1\leq i\leq n$), then
we can set $$A_i = \left(2\|C_i\|\right)^{-1}C_i + \frac{1}{2}I$$
so that $0\leq A_i\leq I$ and thus 
$C_i = 2\|C_i\|(\sum 2^{-k}E_k^{(i)}-\frac{1}{2}I)$. Thus choosing 
$$B_i^{(N)} = 2\|C_i\|\{\sum_{k=1}^N 2^{-k}E_k^{(i)} + 
\sum_{k=N+1}^{\infty}(I-P_k)E_k^{(i)} - \frac{1}{2}I\}$$ one has 
$\|[C_i,B_i^{(N)}]\|_p = 2\|C_i\|\|[A_i,B_i^{(N)}]\|_p\rightarrow 0$
as $N\rightarrow \infty$ for $p\geq n$.
\end{rem}

\section{Spectral Integrals of Operator Functions and Stokes-like Formula}\label{sec: third}
In this section we are going to define spectral integrals of operator functions 
in the next few lemmas.
\begin{lemma}\label{spi1}
Let $H$ be a bounded self-adjoint operator in $\mathcal{H}$ with spectrum in $[a,b]$ and 
let $A:[a,b] \longrightarrow \mathcal{B}(\mathcal{H})$ be operator norm H$\ddot{o}$lder 
continuous with H$\ddot{o}$lder index
$k> \frac{1}{2}$, that is,  
$$\|A(\alpha_1) - A(\alpha_2)\| \leq C|\alpha_1-\alpha_2|^k,$$ \color{black}
where $C$ is some positive constant and $k> \frac{1}{2}.$
Then
$$\int_a^b A(\alpha) E_H(d\alpha)$$
is well-defined as a operator norm Riemann-Stieltjes integral, where $E_H(.)$ is the spectral
measure corresponding to the bounded self-adjoint operator $H$.
\end{lemma}
\begin{proof}
(see also \cite{birsolo0}) Let $P\equiv$ $ \{\Delta_i\}_{i=1}^n$ be a partition of the interval $[a,b]$. Let 
$P'\equiv$ $ \{\Delta_{ij}\}_{[1\leq i\leq n~};$ $_{~1\leq j\leq m]}$
be another partition of $[a,b],$ which is 
finer than $P$ (i.e. $P'\supseteq P$) and is obtained by dividing each interval
$\Delta_i$ of $P$ into equal number (say $m$) of subintervals $\{\Delta_{ij}\}$ 
(i.e. $\Delta_i = \bigcup \limits_{j=1}^m \Delta_{ij}$).
In particular one can use diadic partitions. For $f\in \mathcal{H},$ consider the
following Riemann-Stieltjes sums
$$\Sigma_P f \equiv \sum \limits_{i=1}^n A(\xi_i) E_H(\Delta_i) f
\quad \text{and} \quad \Sigma_{P'}f = \sum \limits_{i=1}^n \sum 
\limits_{j=1}^m A(\xi_{ij}) E_H(\Delta_{ij})f,$$
where $\xi_i \in \Delta_i$ and $\xi_{ij} \in \Delta_{ij}.$ Hence, if we write $|\Delta_i|=$
\emph{length of the interval} $\Delta_i$, we get that
\begin{equation*}
\left(\Sigma_P- \Sigma_{P'}\right) f = \sum \limits_{i=1}^n \sum \limits_{j=1}^m [A(\xi_i) - 
A(\xi_{ij})]E_H(\Delta_{ij})f.
\end{equation*}
Therefore by hypothesis and applying triangle-inequality, 
Cauchy-Schwartz inequality and using the fact that
$|\xi_i-\xi_{ij}| \leq |\Delta_i|,$ we conclude that
\begin{equation*}
\begin{split}
\hspace{0cm} \|\left(\Sigma_P  - \Sigma_{P'}\right) f\| 
\leq \sum_{i=1}^n \sum_{j=1}^m \|A(\xi_i) - A(\xi_{ij})\|\|E_H(\Delta_{ij})f\| 
\leq C \sum_{i=1}^n \sum_{j=1}^m  |\xi_i-\xi_{ij}|^k \|E_H(\Delta_{ij})f\|\\
& \hspace{-15cm} \leq C \left(\sum_{i=1}^n \sum_{j=1}^m |\xi_i-\xi_{ij}|^{2k} \right)^{\frac{1}{2}} 
\left(\sum_{i=1}^n \sum_{j=1}^m \|E_H(\Delta_{ij})f\|^2\right)^{\frac{1}{2}}
= C \|f\| \left(\sum_{i=1}^n \sum_{j=1}^m |\xi_i-\xi_{ij}|^{2k} \right)^{\frac{1}{2}} \\
& \hspace{-15cm} \leq C \|f\| \left(\sum_{i=1}^n \sum_{j=1}^m |\Delta_i|^{2k} \right)^{\frac{1}{2}}
 = C \|f\| m^{\frac{1}{2}} \|P\|^{\left(k-\frac{1}{2}\right)} (b-a)^{\frac{1}{2}},
\end{split}
\end{equation*}
where $\|P\|$ is the norm ($\equiv \max\limits_{1\leq i\leq n}|\Delta_i|$) of the partition 
$P\equiv \{\Delta_i\}_{i=1}^n$. Hence 
\begin{equation*}
\begin{split}
 \|\Sigma_P  - \Sigma_{P'}\| = \sup \limits_{f\in \mathcal{H}~;~f\neq 0} 
 \frac{\|\left(\Sigma_P  - \Sigma_{P'}\right) f\|}{\|f\|}\\
 & \hspace{-4.8cm} \leq C m^{\frac{1}{2}}\|P\|^{k-\frac{1}{2}}
 (b-a)^{\frac{1}{2}}\rightarrow 0 ~~~\text{as} ~~~\|P\|\longrightarrow 0,
 \end{split}
\end{equation*}
proving the existence of the integral $\int \limits_a^b A(\alpha) E_H(d\alpha)$ as a operator norm 
Riemann-Stieltjes 
integral.
\end{proof}
\begin{lemma}\label{spi2}
 Let $A,B,C$ be three bounded self-adjoint operators in an infinite dimensional Hilbert space $\mathcal{H}$ 
such that
$\sigma(A), \sigma(B), \sigma(C) \subseteq [a,b].$ Let $\phi:[a,b]^2\longrightarrow \mathbb{C}$
be a bounded measurable 
function.
Then the symbol $\int_A^B \phi(x,C) dx$, defined as: 
\begin{equation*}
\int_A^B \phi(x,C) dx \equiv \int_a^b \left(\int_a^{\alpha} \phi(x,C) dx \right) [E_B(d\alpha) - E_A(d\alpha)],
\end{equation*}
(where $E_A(.), E_B(.)$ are the spectral measures of the operators $A, B$ respectively), 
exists as a bounded operator.
\end{lemma}
\begin{proof}
 Since $\phi$ is bounded and measurable on $[a,b]^2,$ then by spectral theorem we write
$\phi(x,C) = \int \limits_a^b \phi(x,y) E_C(dy),$ where $E_C(.)$ is the spectral measure 
corresponding to the self-adjoint
operator $C$ and hence $\phi(x,C)$ is a bounded operator with operator norm
\begin{equation*}
 \|\phi(x,C)\| = \sup \{|\phi(x,y)|: ~y\in [a,b]\} = \|\phi\|_{\infty}
\end{equation*}
for all $x\in [a,b]$. Now the map $[a,b] \ni x\longmapsto \phi(x,C) 
\in \mathcal{B}(\mathcal{H})$ is a $\mathcal{B}(\mathcal{H})$-valued 
bounded measurable function and hence the integral $\int \limits_a^{\alpha} \phi(x,C) dx$ exists 
as a Bochner integral
for each fixed $\alpha \in [a,b].$
Moreover, the map $T:[a,b] \longrightarrow \mathcal{B}(\mathcal{H})$ defined by
$T(\alpha) = \int \limits_a^{\alpha} \phi(x,C)dx$ is operator norm H$\ddot{o}$lder continuous 
with H$\ddot{o}$lder
index $1\left(> \frac{1}{2}\right).$ i.e.
\begin{equation*}
 \|T(\alpha_1) - T(\alpha_2)\| \leq \|\phi\|_{\infty} |\alpha_1-\alpha_2| ~~\text{for}~~\alpha_1, 
 \alpha_2 \in [a,b].
\end{equation*}
Thus by Lemma \ref{spi1}, the integrals $\int \limits_a^b T(\alpha) E_A(d\alpha)$ and $\int \limits_a^b 
T(\alpha) E_B(d\alpha)$
exist as a operator norm Riemann-Stieltjes integral. Hence 
$\int \limits_a^b\left(\int \limits_a^{\alpha}\phi(x,C)dx\right)[E_B(d\alpha) - E_A(d\alpha)]$ is
well-defined as a bounded operator and we
denote it by the symbol $\int_A^B \phi(x,C) dx.$
\end{proof}

We derive two formulae for the trace of a Stokes-like expression,
one in terms of a spectral function and the other in terms of divided
differences. First we need a simple lemma.

\begin{lemma}\label{mainsec1}
Let $\psi \in L^{\infty} ([a,b]^2).$ Then there exists two measurable 
functions $\phi_1, \phi_2$ on 
$[a,b]\times [a,b]$ such that $\phi_1$ and $\phi_2$ are differentiable (almost everywhere) with respect to the second
and first variable respectively with bounded derivatives such that,
$$\frac{\partial \phi_2}{\partial x}(x,y) - \frac{\partial \phi_1}{\partial y}(x,y) = \psi(x,y).$$
Moreover $\phi_1$ and $\phi_2$ are Lipschitz in the second and first variable
respectively, uniformly with respect to the other variable.
Conversely, if $\phi_1$ and $\phi_2$ are two measurable functions
differentiable with respect to the second and first
variable respectively with bounded measurable derivatives, then
$$\psi(x,y) = \frac{\partial \phi_2}{\partial x}(x,y) - 
\frac{\partial \phi_1}{\partial y}(x,y) \in L^{\infty}([a,b]^2).$$
\end{lemma}
\begin{proof}
Let $\psi \in L^{\infty}([a,b]^2)$ and $\phi_1$, $\phi_2$ be defined as:
\begin{equation}\label{phipsieq}
\begin{split}
&\phi_1(x,y) = -\frac{1}{2} \int\limits_a^y \psi(x,t)dt + \psi_1(x)
= \widetilde{\phi}_1(x,y) + \psi_1(x)~~\text{and}\\
&\phi_2(x,y) = \frac{1}{2} \int\limits_a^x \psi(t,y)dt + \psi_2(y)
= \widetilde{\phi}_2(x,y) + \psi_2(y),
\end{split}
\end{equation}
where $\psi_1$, $\psi_2$ are two measurable functions on $[a,b].$ 
Thus from the definition of $\phi_1$ and $\phi_2$ it follows that $\phi_1, \phi_2$ are two 
measurable functions, differentiable with respect to the second and first variable respectively with 
$\frac{\partial \phi_1}{\partial y}(x,y) = -\frac{1}{2} \psi(x,y)$ and 
$\frac{\partial \phi_2}{\partial x}(x,y) = \frac{1}{2} \psi(x,y)$
almost everywhere and hence $\frac{\partial \phi_2}{\partial x}(x,y) 
- \frac{\partial \phi_1}{\partial y}(x,y) = \psi(x,y).$
Moreover, $\left|\frac{\partial \phi_1}{\partial y}(x,y) \right|, 
\left|\frac{\partial \phi_2}{\partial x}(x,y)\right| \leq \frac{1}{2} \|\psi\|_{\infty}.$
Again from the definition of $\phi_1,$ we conclude that $\phi_1(x,y_1) 
- \phi_1(x,y_2) = \int \limits_{y_2}^{y_1} \psi(x,t)dt,$ which
implies that $\left|\phi_1(x,y_1) - \phi_1(x,y_2)\right|\leq \|\psi\|_{\infty} |y_1-y_2|$ 
and hence $\phi_1$ is Lipschitz 
in second variable uniformly with respect to the first variable.
By repeating the above argument we conclude that $\phi_2$ 
is also Lipschitz in first variable uniformly with respect to the second variable.
\vspace{0.1in}

Converse part follows from the hypotheses.

\end{proof}

The following theorem tells us about the trace formula for two variables in finite dimension.

\begin{thm}\label{mainsec2}
Let $P$ and $Q$ be two finite dimensional projections in $\mathcal{H}$ and let $\left(H_1^0,H_2^0\right)$
and $\left(H_1,H_2\right)$ be two commuting
pairs of self-adjoint operators acting in the reducing subspaces $P\mathcal{H}$
and $Q\mathcal{H}$ respectively. Let
$$\psi \in L^{\infty}([a,b]^2),$$ 
where
$$\sigma\left(H_1\right),
\sigma \left(H_2\right),
\sigma\left(H_1^0\right), \sigma\left(H_2^0\right)
\subseteq [a,b].$$ Then
\begin{equation}
\begin{split}
\hspace{0cm} \mathcal{I} \equiv\textup{Tr}\{\int_{H_1^0}^{H_1} P\phi_1\left(x,H_2^0\right)Q dx + 
\int_{H_2^0}^{H_2} Q\phi_2\left(H_1,y\right)Pdy \\
& \hspace{-5.5cm} + \int_{H_1}^{H_1^0} P \phi_1\left(x,H_2\right)Q dx
+ \int_{H_2}^{H_2^0}Q \phi_2\left(H_1^0,y\right)P dy\} \\
& \hspace{-9.4cm} = \textup{Tr}\{\int_{H_1^0}^{H_1} P\left[\phi_1\left(x,H_2^0\right)
- \phi_1\left(x,H_2\right)\right]Q dx \\
&\hspace{-4.5cm} + \int_{H_2^0}^{H_2} Q\left[\phi_2\left(H_1,y\right)
- \phi_2\left(H_1^0,y\right)\right]P dy\}\\
& \hspace{-9.4cm} = \int_a^b \int_a^b \left[\frac{\partial \phi_2}{\partial x}(x,y) -
\frac{\partial \phi_1}{\partial y}(x,y)\right]\xi (x,y) dx dy
= \int_a^b \int_a^b \psi(x,y) \xi (x,y) dx dy,
\end{split}
\end{equation}
where 
\begin{equation*} 
\xi (x,y) =  \textup{Tr}\{Q\left[E_{H_1}(x)  - E_{H_1^0}(x)\right]
P\left[E_{H_2}(y) - E_{H_2^0}(y)\right]Q\}
\end{equation*}
and $E_{H_1}(.)$, 
$ E_{H_2}(.)$, $ E_{H_1^0}(.)$, $E_{H_2^0}(.)$ are the spectral measures of the 
operators $H_1$, $H_2$, $H_1^0$, $H_2^0$ respectively and 
$\phi_1$, $\phi_2$ are same as in \eqref{phipsieq}.
\end{thm}
\begin{proof}
Let  $\psi \in L^{\infty}([a,b]^2)$ and $\phi_1$, $\phi_2$ be defined as in \eqref{phipsieq}.
We note that
\begin{equation*}
 \begin{split}
& \phi_1\left(x,H_2^0\right)
- \phi_1\left(x,H_2\right) = \widetilde{\phi}_1\left(x,H_2^0\right)
- \widetilde{\phi}_1\left(x,H_2\right),\\
& \phi_2\left(H_1,y\right)
- \phi_2\left(H_1^0,y\right) = \widetilde{\phi}_2\left(H_1,y\right)
- \widetilde{\phi}_2\left(H_1^0,y\right);~~\text{and that}
\end{split}
\end{equation*}
$$ \frac{\partial \phi_2}{\partial x}(x,y) - \frac{\partial \phi_1}{\partial y}(x,y) 
= \frac{\partial \widetilde{\phi}_2}{\partial x}(x,y) - \frac{\partial \widetilde{\phi}_1}{\partial y}(x,y)
= \psi(x,y)~~\text{almost everywhere}.$$
By Lemma \ref{mainsec1}, $\widetilde{\phi}_1(x,.)$ is Lipschitz uniformly
with respect to $x$ and $\widetilde{\phi}_2(.,y)$ is Lipschitz uniformly with respect to $y$ and
we note that
by Lemma \ref{spi1}, the integrals
\begin{equation*}
\begin{split}
\hspace{0.1cm}  \int_{H_1^0}^{H_1} P\left[\phi_1\left(x,H_2^0\right)
- \phi_1\left(x,H_2\right)\right]Q~dx~~\text{and}  \\
& \hspace{-11.3cm} \int_{H_2^0}^{H_2} Q\left[\phi_2\left(H_1,y\right)
- \phi_2\left(H_1^0,y\right)\right]P~dy~~\text{exist in} ~~\mathcal{B}(\mathcal{H}).
\end{split}
\end{equation*}
Thus
\begin{equation}
\begin{split}
\hspace{-1.5cm} \textup{Tr} \{\int_{H_1^0}^{H_1} P\left[\phi_1\left(x,H_2^0\right)
- \phi_1\left(x,H_2\right)\right]Q~dx\} 
= \textup{Tr} \{ \int_{H_1^0}^{H_1} P
\left[\widetilde{\phi}_1\left(x,H_2^0\right)
- \widetilde{\phi}_1\left(x,H_2\right)\right]Q~dx\} \\
& \hspace{-15cm} = \textup{Tr}\{ \int\limits_a^b \left(\int\limits_a^{\alpha} P\left[\widetilde{\phi}_1
\left(x,H_2^0\right)
- \widetilde{\phi}_1\left(x,H_2\right)\right]Q~dx\right)\left[E_{H_1}(d\alpha) 
- E_{H_1^0}(d\alpha)\right]\}.
\end{split}
\end{equation}
\begin{equation*}
\begin{split}
& \hspace{-0.2cm} = \textup{Tr}\{ \int\limits_a^b \left(\int\limits_a^{\alpha} P
\left[\widetilde{\phi}_1\left(x,H_2^0\right)
- \widetilde{\phi}_1\left(x,H_2\right)\right]Q~dx\right)Q\left[E_{H_1}(d\alpha) 
- E_{H_1^0}(d\alpha)\right]P\}\\
\end{split}
\end{equation*}
\begin{equation*}
\begin{split}
& \hspace{-0.4cm} = \textup{Tr}\{\left(\int\limits_a^{\alpha} P
\left[\widetilde{\phi}_1\left(x,H_2^0\right)
- \widetilde{\phi}_1\left(x,H_2\right)\right]Q~dx\right)Q[E_{H_1}(\alpha)
 - E_{H_1^0}(\alpha)]P|_{\alpha=a}^b\\
& \hspace{2cm} - \int\limits_a^b P\left[\widetilde{\phi}_1\left(\alpha,H_2^0\right)
- \widetilde{\phi}_1\left(\alpha,H_2\right)\right]Q\left[E_{H_1}(\alpha)
 - E_{H_1^0}(\alpha)\right] Pd\alpha \}\\
\end{split}
\end{equation*}
\begin{equation*}
\begin{split}
& \hspace{-1.6cm} = - \textup{Tr}\{ \int\limits_a^b P \left[\widetilde{\phi}_1\left(\alpha,H_2^0\right)
-  \widetilde{\phi}_1\left(\alpha,H_2\right)\right]Q\left[E_{H_1}(\alpha)
 - E_{H_1^0}(\alpha)\right]P d\alpha \},
\end{split}
\end{equation*}
where we have integrated by-parts and used the fact that the boundary terms vanish since 
$\int\limits_a^b P \left[\widetilde{\phi}_1\left(x,H_2^0\right)
-  \widetilde{\phi}_1\left(x,H_2\right)\right]Qdx\in \mathcal{B}(\mathcal{H})$ as a Bochner integral.
Furthermore by integrating by parts in the $\beta$-integral and noting that $\widetilde{\phi}_1(\alpha,.)$
is continuous for almost all $\alpha$ fixed, the above expression is equal to
\begin{equation*}
\begin{split}
& \hspace{-0.2cm} = -\textup{Tr} \{\int\limits_a^b d\alpha \left(\int\limits_a^b \widetilde{\phi}_1(\alpha,\beta)
P\left[E_{H_2^0}(d\beta) -  E_{H_2}(d\beta)\right]Q\right)\left[E_{H_1}(\alpha)
 - E_{H_1^0}(\alpha)\right]P\}\\
\end{split}
\end{equation*}
\begin{equation*}
\begin{split}
& \hspace{0cm} = -\textup{Tr} \int\limits_a^b d\alpha ~\{ ~\widetilde{\phi}_1(\alpha,\beta)
P\left[E_{H_2^0}(\beta) -  E_{H_2}(\beta)\right]Q|_{\beta =a}^b\\
& \hspace{2.5cm} - \int\limits_a^b \frac{\partial\widetilde{\phi}_1}{\partial \beta}(\alpha,\beta)
P\left[E_{H_2^0}(\beta) - E_{H_2}(\beta)\right]Q ~d\beta\}~
Q~\left[E_{H_1}(\alpha) - E_{H_1^0}(\alpha)\right]P\\
\end{split}
\end{equation*}
\begin{equation*}
\begin{split}
& \hspace{-0.7cm} = -\textup{Tr} \{ \int\limits_a^b \int\limits_a^b \frac{\partial\widetilde{\phi}_1}
{\partial \beta}(\alpha,\beta)
~P\left[E_{H_2}(\beta) - E_{H_2^0}(\beta)\right]Q
\left[E_{H_1}(\alpha) - E_{H_1^0}(\alpha)\right]P d\alpha d\beta\}\\
& \hspace{-0.7cm} = -\int\limits_a^b \int\limits_a^b \frac{\partial\widetilde{\phi}_1}{\partial \beta}(\alpha,\beta)
~\textup{Tr}\{ Q\left[E_{H_1}(\alpha) - E_{H_1^0}(\alpha)\right]
P\left[E_{H_2}(\beta) - E_{H_2^0}(\beta)\right]Q\} ~d\alpha d\beta,\\
\end{split}
\end{equation*}
since by Fubini's theorem, the iterated integral is equal to the double integral in this case.
Thus we have
\begin{equation}\label{fin1}
\begin{split}
\textup{Tr} \{\int_{H_1^0}^{H_1} P\left[\phi_1\left(x,H_2^0\right)
- \phi_1\left(x,H_2\right)\right]Q~dx\} \\
& \hspace{-8cm} =  -\int\limits_a^b \int\limits_a^b \frac{\partial\widetilde{\phi}_1}{\partial y}(x,y)
~\textup{Tr}\{ Q\left[E_{H_1}(x) - E_{H_1^0}(x)\right]
P\left[E_{H_2}(y) - E_{H_2^0}(y)\right]Q\} ~dx dy.
\end{split}
\end{equation}
By an identical set of computations as above, we have that
\begin{equation}\label{fin2}
\begin{split}
\textup{Tr} \{\int_{H_2^0}^{H_2} Q\left[\phi_2\left(H_1,y\right)
- \phi_2\left(H_1^0,y\right)\right]P~dy\} \\
&\hspace{-8cm} = \textup{Tr} \{\int_a^b \left(\int_a^{\alpha} 
Q\left[\widetilde{\phi}_2\left(H_1,y\right)
- \widetilde{\phi}_2\left(H_1^0,y\right)\right]P~dy\right)
\left[E_{H_2}(d\alpha)-E_{H_2^0}(d\alpha)\right]\}\\
& \hspace{-8cm} =   \int\limits_a^b \int\limits_a^b \frac{\partial\widetilde{\phi}_2}{\partial x}(x,y)
~\textup{Tr}\{ Q\left[E_{H_1}(x) - E_{H_1^0}(x)\right]
P\left[E_{H_2}(y) - E_{H_2^0}(y)\right]Q\} ~dx dy.
\end{split}
\end{equation}
Combining \eqref{fin1} and \eqref{fin2}, we get
\begin{equation*}
\begin{split}
\mathcal{I}\equiv \textup{Tr}\{\int_{H_1^0}^{H_1} P\left[\phi_1\left(x,H_2^0\right)
- \phi_1\left(x,H_2\right)\right]Q dx \\
&\hspace{-4cm} + \int_{H_2^0}^{H_2} Q\left[\phi_2\left(H_1,y\right)
- \phi_2\left(H_1^0,y\right)\right]Pdy\}\\
& \hspace{-7.7cm} = \int\limits_a^b \int\limits_a^b \left[\frac{\partial \widetilde{\phi}_2}{\partial x}(x,y) -
\frac{\partial \widetilde{\phi}_1}{\partial y}(x,y)\right]\xi(x,y) dx dy
= \int\limits_a^b \int\limits_a^b \psi(x,y) \xi(x,y) dx dy, 
\end{split}
\end{equation*}
where $\xi(x,y) = \textup{Tr}\{Q\left[E_{H_1}(x)  - E_{H_1^0}(x)\right]
P\left[E_{H_2}(y) - E_{H_2^0}(y)\right]Q\}.$
\end{proof}

The following theorem 
finds another formula for the above Stokes-like expression $\mathcal{I}$
of operator functions under trace in terms of
divided differences, which is useful to control the measure generated by $\xi$.  
\begin{thm}\label{dividthm}
Under the hypotheses of Theorem \ref{mainsec2},
\begin{equation*}
\mathcal{I} 
 = \int\limits_{[a,b]^2} \int\limits_{[a,b]^2} \frac{\int\limits_{x_2}^{x_1}
\int\limits_{y_2}^{y_1} \psi(x,y)dx dy}{(x_1-x_2)(y_1-y_2)}
~\left\langle \left(H_1-H_1^0\right), PE_{\underline{H}^0}(dx_2\times dy_1)
\left(H_2-H_2^0\right)E_{\underline{H}}(dx_1\times dy_2)Q\right\rangle_2,
\end{equation*}
where $\underline{H}^0 = (H_1^0,H_2^0)$, $\underline{H} = (H_1,H_2)$ 
and $E_{\underline{H}^0}(.)$ and $E_{\underline{H}}(.)$ are the spectral measures of the 
operators tuples $\underline{H}^0$ and $\underline{H}$ respectively on the Borel
sets of $[a,b]^2$ and where $\left\langle .,.\right\rangle_2$ denotes the inner product 
of the Hilbert space $\mathcal{B}_2(\mathcal{H})$.
\end{thm}
\begin{proof}
In finite dimensional space $\mathcal{H}$, using the ideas of double spectral integrals
(\cite{birsolo1},\cite{birsolo2}), we get that

\begin{equation*}
\begin{split}
 P\left[\widetilde{\phi}_1(x,H_2^0) - \widetilde{\phi}_1(x,H_2)\right]Q
 = \int\limits_a^b \left[ \widetilde{\phi}_1(x,y_1)
 - \widetilde{\phi}_1(x,y_2)\right] PE_{H_2^0}(dy_1) 
 E_{H_2}(dy_2)Q\\
 & \hspace{-12cm} = -\int\limits_a^b \int\limits_a^b ~\frac{\widetilde{\phi}_1(x,y_1) - 
 \widetilde{\phi}_1(x,y_2)}{y_1-y_2}~PE_{H_2^0}(dy_1)\left(H_2-H_2^0\right) E_{H_2}(dy_2)Q\\
 & \hspace{-12cm} = \frac{1}{2} \int\limits_{[a,b]^2} ~~\frac{\int\limits_{y_2}^{y_1} 
 \psi(x,t)dt}{y_1-y_2}~PE_{H_2^0}(dy_1)\left(H_2-H_2^0\right) E_{H_2}(dy_2)Q.
\end{split}
\end{equation*}
Therefore 
\begin{equation*}
\begin{split}
&\int\limits_a^b \left(\int\limits_a^{\alpha} P\left[\widetilde{\phi}_1(x,H_2^0) - 
\widetilde{\phi}_1(x,H_2)\right]Qdx\right)\left[E_{H_1}(d\alpha)
- E_{H_1^0}(d\alpha)\right]\\
&= \frac{1}{2} \int\limits_a^b \left(\int\limits_a^{\alpha}
\left[\int\limits_{[a,b]^2} ~~\frac{\int\limits_{y_2}^{y_1} 
 \psi(x,t)dt}{y_1-y_2}~PE_{H_2^0}(dy_1)\left(H_2-H_2^0\right) E_{H_2}(dy_2)Q\right]dx\right)
 \left[E_{H_1}(d\alpha)
- E_{H_1^0}(d\alpha)\right]\\
& = \frac{1}{2}\left[ \int\limits_{[a,b]^2} ~PE_{H_2^0}(dy_1)\left(H_2-H_2^0\right) E_{H_2}(dy_2)Q
\right]\int\limits_a^b \int\limits_a^{\alpha} ~~\frac{\int\limits_{y_2}^{y_1} 
 \psi(x,t)dt}{y_1-y_2}dx \left[E_{H_1}(d\alpha)
- E_{H_1^0}(d\alpha)\right]
\end{split}
\end{equation*}

\begin{equation*}
\begin{split}
& = \frac{1}{2}\left[ \int\limits_{[a,b]^2} ~PE_{H_2^0}(dy_1)\left(H_2-H_2^0\right) E_{H_2}(dy_2)Q
\right]\int\limits_{[a,b]^2} \frac{\int\limits_{\beta}^{\alpha} ~~\frac{\int\limits_{y_2}^{y_1} 
 \psi(x,t)dt}{y_1-y_2}dx}{\alpha-\beta} \left[E_{H_1}(d\alpha)
\left(H_1-H_1^0\right)E_{H_1^0}(d\beta)\right],
\end{split}
\end{equation*}
where we have used Fubini's theorem to interchange the order of the integration and utilized 
the ideas of double spectral integrals
(\cite{birsolo1},\cite{birsolo2}). Thus by using trace properties, Fubini's theorem 
and the fact that $(H_1^0,H_2^0)$, $(H_1,H_2)$ are two commuting pairs of self-adjoint
operators, we conclude that

\begin{equation}\label{divid1}
\begin{split}
 &\textup{Tr} \{\int_{H_1^0}^{H_1} P\left[\phi_1\left(x,H_2^0\right)
- \phi_1\left(x,H_2\right)\right]Q~dx\}\\
& = \frac{1}{2} \textup{Tr}\{\left[ \int\limits_{(y_1,y_2)\in [a,b]^2} 
~PE_{H_2^0}(dy_1)\left(H_2-H_2^0\right) E_{H_2}(dy_2)Q
\right]\bullet\\
& \hspace{5cm} \int\limits_{(x_1,x_2)\in [a,b]^2}~~\frac{\int\limits_{x_2}^{x_1}\int\limits_{y_2}^{y_1} 
 \psi(x,t)dtdx}{(x_1-x_2)(y_1-y_2)}\left[E_{H_1}(dx_1)
\left(H_1-H_1^0\right)E_{H_1^0}(dx_2)\right]\}\\
\end{split}
\end{equation}
\vspace{0.03in}

\begin{equation*}
\begin{split}
& = \frac{1}{2} \int\limits_{[a,b]^2}\int\limits_{[a,b]^2}
\left[\frac{\int\limits_{x_2}^{x_1}\int\limits_{y_2}^{y_1} 
 \psi(x,y)dxdy}{(x_1-x_2)(y_1-y_2)}\right]\bullet\\
 & \hspace{4cm} \textup{Tr} \{PE_{H_2^0}(dy_1)\left(H_2-H_2^0\right) E_{H_2}(dy_2)Q
 E_{H_1}(dx_1)
\left(H_1-H_1^0\right)E_{H_1^0}(dx_2)\}\\
\end{split}
\end{equation*}
\begin{equation*}
\begin{split}
& = \frac{1}{2} \int\limits_{[a,b]^2}\int\limits_{[a,b]^2}
\left[\frac{\int\limits_{x_2}^{x_1}\int\limits_{y_2}^{y_1} 
 \psi(x,y)dxdy}{(x_1-x_2)(y_1-y_2)}\right]\bullet\\
 & \hspace{4cm} \textup{Tr} \{\left(H_1-H_1^0\right)PE_{H_1^0}(dx_2) E_{H_2^0}(dy_1)
 \left(H_2-H_2^0\right) E_{H_1}(dx_1) E_{H_2}(dy_2)Q\}\\
 \end{split}
\end{equation*}
\begin{equation*}
\begin{split}
 & = \frac{1}{2} \int\limits_{[a,b]^2}\int\limits_{[a,b]^2}
\left[\frac{\int\limits_{x_2}^{x_1}\int\limits_{y_2}^{y_1} 
 \psi(x,y)dxdy}{(x_1-x_2)(y_1-y_2)}\right]\bullet\\
 & \hspace{4cm} \left\langle\left(H_1-H_1^0\right), PE_{H_1^0}(dx_2) E_{H_2^0}(dy_1)
 \left(H_2-H_2^0\right) E_{H_1}(dx_1) E_{H_2}(dy_2)Q\right\rangle_2
\end{split}
\end{equation*}
\vspace{0.1in}

\begin{equation*}
\begin{split}
& = \frac{1}{2} \int\limits_{[a,b]^2}\int\limits_{[a,b]^2}
\left[\frac{\int\limits_{x_2}^{x_1}\int\limits_{y_2}^{y_1} 
 \psi(x,y)dxdy}{(x_1-x_2)(y_1-y_2)}\right]\bullet\\
 & \hspace{4cm} \left\langle \left(H_1-H_1^0\right), PE_{\underline{H}^0}(dx_2\times dy_1)
\left(H_2-H_2^0\right)E_{\underline{H}}(dx_1\times dy_2)Q\right\rangle_2,
\end{split}
\end{equation*}
where $E_{\underline{H}^0}(dx_2\times dy_1) = E_{H_1^0}(dx_2) E_{H_2^0}(dy_1)$ and 
$E_{\underline{H}}(dx_1\times dy_2) = E_{H_1}(dx_1) E_{H_2}(dy_2)$ are
$\mathcal{B}_2(\mathcal{H})$-valued spectral measure on $[a,b]^2$.
Similarly by an identical set of computations, we get that
\begin{equation}\label{divid2}
\begin{split}
&\textup{Tr} \{\int_{H_2^0}^{H_2} Q\left[\phi_2\left(H_1,y\right)
- \phi_2\left(H_1^0,y\right)\right]P~dy\}\\
& = \frac{1}{2} \int\limits_{[a,b]^2}\int\limits_{[a,b]^2}
\left[\frac{\int\limits_{x_2}^{x_1}\int\limits_{y_2}^{y_1} 
 \psi(x,y)dxdy}{(x_1-x_2)(y_1-y_2)}\right]\bullet\\
 & \hspace{4cm} \left\langle \left(H_1-H_1^0\right), PE_{\underline{H}^0}(dx_2\times dy_1)
\left(H_2-H_2^0\right)E_{\underline{H}}(dx_1\times dy_2)Q\right\rangle_2.
\end{split}
\end{equation}
The conclusion of the theorem follows by combining the equations \eqref{divid1} and \eqref{divid2}.

\end{proof}
The next theorem shows how Theorem \ref{th: WBG} can be used for the reduction to finite dimension. 
We deal here with the simpler case of $n=2.$ In the statement of the theorem below we apply Theorem 
\ref{th: WBG} to the pairs $(H_1^0,H_2^0)$ and $(H_1,H_2)$ to get two
commuting pairs of finite dimensional self-adjoint operators 
$\left(H_1^{0(N)},H_2^{0(N)}\right)$ and $\left(H_1^{(N)},H_2^{(N)}\right)$
in $P_N^0\mathcal{H}$ and $P_N\mathcal{H}$ respectively, such that 
\begin{equation}\label{ap1}
\left\|[H_j^0,P_N^0]\right\|_p,~\left\|P_N^0H_j^0P_N^0 - H_j^{0(N)}P_N^0\right\|_p \longrightarrow 0 ~\text{as}~
N\longrightarrow \infty ~\text{for}~ p\geq 2,~ j = 1,2~~\text{and} 
\end{equation}
\begin{equation}\label{ap2}
 \hspace{-0.6cm} \left\|[H_j,P_N]\right\|_p,~\left\|P_NH_jP_N - H_j^{(N)}P_N\right\|_p 
\longrightarrow 0 ~\text{as}~
N\longrightarrow \infty ~\text{for}~ p\geq 2,~ j = 1,2, 
\end{equation}
where $P_N^0$, $P_N$ are projections increasing to $I$ (i.e. $P_N^0, P_N \uparrow I$).

\begin{thm}\label{appth1}
Let $(H_1^0,H_2^0)$ and $(H_1,H_2)$ be two commuting pairs of bounded self-adjoint operators in a separable Hilbert
space $\mathcal{H}$ such that $H_j-H_j^0 \equiv V_j\in \mathcal{B}_2(\mathcal{H})$ 
and such that $\sigma(H_j)$, $\sigma(H_j^0)$ $\subseteq [a,b] $ for $j=1,2.$ Furthermore let
$$p_1(x,y) = \sum_{0\leq i+j\leq n} c(i,j) x^iy^j,~
p_2(x,y) = \sum_{0\leq r+s\leq m} d(r,s) x^ry^s$$ be two polynomials in 
$[a,b]^2$ with complex coefficients. Then 
\begin{equation}\label{fappthm}
\begin{split}
\hspace{-0.5cm} \textup{Tr} \{\int_{H_1^0}^{H_1} p_1(x,H_2^0) dx + \int_{H_2^0}^{H_2} p_2(H_1,y)dy \\
&\hspace{-3cm} + \int_{H_1}^{H_1^0} p_1(x,H_2)dx
+ \int_{H_2}^{H_2^0} p_2(H_1^0,y)dy\}\\
& \hspace{-6.5cm} = \lim_{N\longrightarrow \infty} \textup{Tr}\{\int_{H_1^{0(N)}}^{H_1^{(N)}} 
P_N^0\left[p_1\left(x,H_2^{0(N)}\right)-p_1\left(x,H_2^{(N)}\right)\right]P_N dx  \\
& \hspace{-2.5cm} + \int_{H_2^{0(N)}}^{H_2^{(N)}} P_N\left[p_2\left(H_1^{(N)},y\right)
- p_2\left(H_1^{0(N)},y\right)\right]P_N^0dy\}\\
& \hspace{-6.5cm} = \lim_{N\longrightarrow \infty}
\int\limits_a^b \int\limits_a^b \left[\frac{\partial p_2}{\partial x}(x,y) -
\frac{\partial p_1}{\partial y}(x,y)\right]\xi_N(x,y) dx dy,
\end{split}
\end{equation}
where $\xi_N(x,y) = \textup{Tr}\{P_N\left[E_{H_1^{(N)}}(x)  - E_{H_1^{0(N)}}(x)\right]
P_N^0\left[E_{H_2^{(N)}}(y) - E_{H_2^{0(N)}}(y)\right]P_N\}$
and $E_{H_1^{0(N)}}(.)$, 
$E_{H_2^{0(N)}}(.)$, $E_{H_1^{(N)}}(.)$, $E_{H_2^{(N)}}(.)$ are the spectral measures of the 
operators $H_1^{0(N)}$, $H_2^{0(N)}$, $H_1^{(N)}$, $H_2^{(N)}$ respectively
and $P_N^0$, $P_N$ are the projections obtained by applying Theorem 
\ref{th: WBG} to the pairs $(H_1^0,H_2^0)$ and $(H_1,H_2)$ respectively, as mentioned above.
\end{thm}
We need a preliminary result for the proof of the above theorem.
\begin{lemma}\label{plmmaappthm}
 Let $H_j^0$, $H_j$, $H_j^{0(N)}$, $H_j^{(N)}$, $P_N^0$, $P_N$ be as above for $j=1,2$. Then 
 \vspace{0.1in}
 
 $(i)$ $\left\|P_N^0\left(H_j^0\right)^k-\left(P_N^0H_j^0P_N^0\right)^k\right\|_2$ and  
 $\left\|P_N\left(H_j\right)^k-\left(P_NH_jP_N\right)^k\right\|_2$ $\longrightarrow 0$
 as $N\longrightarrow \infty$ for $j=1,2$ and $k\geq 1$.
 \vspace{0.1in}
 
 $(ii)$ $\left\|\left(H_j^0\right)^kP_N^0-\left(P_N^0H_j^0P_N^0\right)^k\right\|_2$ and 
 $\left\|\left(H_j\right)^kP_N-\left(P_NH_jP_N\right)^k\right\|_2$ $\longrightarrow 0$
 as $N\longrightarrow \infty$ for $j=1,2$ and $k\geq 1$.
 \vspace{0.1in}
 
 $(iii)$ $\left\|P_N^0\left[\left(P_N^0H_j^0P_N^0\right)^k - 
 \left(P_NH_jP_N\right)^k\right]P_N\right\|_2$ is uniformly bounded in $N$ for $j=1,2$ and
 $k\geq 1$.
 \vspace{0.1in}
 
 $(iv)$ $\lim\limits_{N\longrightarrow \infty} \textup{Tr}\{P_N^0\left[\left(H_2^0\right)^k-
 \left(H_2\right)^k\right]P_N\left[\left(H_1\right)^l-\left(H_1^0\right)^l\right]P_N^0\}$
 
 $=$ $\lim\limits_{N\longrightarrow \infty} \textup{Tr}\{P_N^0\left[\left(H_2^{0(N)}\right)^k-
 \left(H_2^{(N)}\right)^k\right]P_N\left[\left(H_1^{(N)}\right)^l-\left(H_1^{0(N)}\right)^l\right]P_N^0\}$
 for $k, l\geq 1$.
 \vspace{0.1in}
 
 $(v)$ $\left\|P_N\left(H_j^{(N)}-H_j^{0(N)}\right)P_N^0
 - P_NV_jP_N^0\right\|_2\longrightarrow 0$ as $N\longrightarrow \infty$ for $j=1,2$.
 
\end{lemma}
\begin{proof}
 $(i)$ Note that 
 \begin{equation*}
 \begin{split}
   & P_N^0\left(H_j^0\right)^k-\left(P_N^0H_j^0P_N^0\right)^k 
   = P_N^0\left(H_j^0\right)^kP_N^0 -\left(P_N^0H_j^0P_N^0\right)^k + 
   P_N^0\left(H_j^0\right)^k\left(P_N^0\right)^{\bot}\\
   & =-P_N^0\sum_{l=0}^{k-1} \left(P_N^0H_j^0P_N^0\right)^{k-l-1}\left[
   P_N^0H_j^0P_N^0 - H_j^0\right]\left(H_j^0\right)^lP_N^0 + 
   P_N^0\left(H_j^0\right)^k\left(P_N^0\right)^{\bot}\\
   &= - P_N^0\sum_{l=0}^{k-1} \left(P_N^0H_j^0P_N^0\right)^{k-l-1}
   P_N^0\left[H_j^0,P_N^0\right]\left(H_j^0\right)^lP_N^0 +
   P_N^0\left[P_N^0, \left(H_j^0\right)^k\right]
 \end{split}
\end{equation*}
and therefore
\begin{equation*}
\begin{split}
 \hspace{-5cm} \left\|P_N^0\left(H_j^0\right)^k-\left(P_N^0H_j^0P_N^0\right)^k\right\|_2
 \leq k\left\|H_j^0\right\|^{k-1}\left\|\left[H_j^0,P_N^0\right]\right\|_2
 + \left\|\left[P_N^0, \left(H_j^0\right)^k\right]\right\|_2\\
 & \hspace{-8cm} \longrightarrow 0~~\text{as}~~N\longrightarrow \infty,
 \end{split}
\end{equation*}
by using \eqref{ap1} for $j=1,2$. By similar computations as above and using the equation \eqref{ap2}
we achieve the other conclusion of $(i)$.
\vspace{0.1in}

$(ii)$ The proof of $(ii)$ follows from $(i)$ by taking the adjoint of the 
expressions involved.
\vspace{0.1in}

$(iii)$ By $(i)$ and $(ii)$ we conclude that 
\begin{equation*}
 \left\|P_N^0\{\left[ \left(P_N^0H_j^0P_N^0\right)^k - 
 \left(P_NH_jP_N\right)^k\right] - \left[\left(H_j^0\right)^k - 
 \left(H_j\right)^k\right]\}P_N\right\|_2\longrightarrow 0 
\end{equation*}
as $N\longrightarrow \infty$ and this completes the proof of $(iii)$
since $\left(H_j^0\right)^k - \left(H_j\right)^k\in \mathcal{B}_2(\mathcal{H})$
for $j=1,2$ and $k\geq 1$.
\vspace{0.1in}

$(iv)$ By $(i)$ and $(ii)$ we have for $k, l\geq 1$ that
\begin{equation*}
\begin{split}
 &\lim\limits_{N\longrightarrow \infty} \textup{Tr}\{P_N^0\left[\left(H_2^0\right)^k-
 \left(H_2\right)^k\right]P_N\left[\left(H_1\right)^l-\left(H_1^0\right)^l\right]P_N^0\} \\
 & = \lim\limits_{N\longrightarrow \infty} \textup{Tr}\{P_N^0\left[\left(P_N^0H_2^0P_N^0\right)^k-
 \left(P_NH_2P_N\right)^k\right]P_N\left[\left(P_NH_1P_N\right)^l-\left(P_N^0H_1^0P_N^0\right)^l
 \right]P_N^0\}.
\end{split}
\end{equation*}
The conclusion of $(iv)$ follows from the above equality and using 
the equations \eqref{ap1} and \eqref{ap2}.
\vspace{0.1in}

$(v)$ First note that
\begin{equation*}
\begin{split}
 & P_N\left(H_j^{(N)}-H_j^{0(N)}\right)P_N^0
 - P_N\left(H_j-H_j^0\right) P_N^0 \\
 & = P_N\left(P_NH_j^{(N)} - P_NH_jP_N\right)P_N^0 
 - P_N\left[P_N,H_j\right]P_N^0 - P_N\left(H_j^{0(N)}P_N^0 - P_N^0H_j^0P_N^0\right)\\
 &\hspace{4cm}  + P_N\left[H_j^0,P_N^0\right]P_N^0
\end{split}
\end{equation*}
and therefore 
\begin{equation*}
\begin{split}
 & \left\|P_N\left(H_j^{(N)}-H_j^{0(N)}\right)P_N^0
 - P_N\left(H_j-H_j^0\right) P_N^0 \right\|_2\\
 & \leq \left\|P_NH_j^{(N)} - P_NH_jP_N\right\|_2
 + \left\|\left[P_N,H_j\right]\right\|_2 + \left\|H_j^{0(N)}P_N^0 - P_N^0H_j^0P_N^0\right\|_2
 +\left\|\left[H_j^0,P_N^0\right]\right\|_2,
\end{split}
\end{equation*}
which converges to $0$ as $N\longrightarrow \infty$ for $j=1,2$ by using
the equations \eqref{ap1} and \eqref{ap2}.
\end{proof}

\textbf{Proof of Theorem \ref{appth1}}:
Note that 
\begin{equation*}
\begin{split}
 &\int_{H_1^0}^{H_1} \left[p_1(x,H_2^0) - p_1(x,H_2)\right] dx \\
 & = \sum\limits_{2\leq i+j\leq n ~;~ i,j \in \mathbb{N}} \frac{c(i,j)}{i+1}
\left[(H_2^0)^j-(H_2)^j\right]\left[(H_1)^{i+1}-(H_1^0)^{i+1}\right],
\end{split}
\end{equation*}
which is trace class by hypotheses. Similarly,
\begin{equation*}
\begin{split}
 & \int_{H_2^0}^{H_2}\left[p_2(H_1,y) - p_2(H_1^0,y) \right]dy
  \\
 & = \sum\limits_{2\leq r+s\leq m~;~r,s \in\mathbb{N}} \frac{d(r,s)}{s+1}\left[(H_1)^r-(H_1^0)^r\right]
\left[(H_2)^{s+1}-(H_2^0)^{s+1}\right] \in \mathcal{B}_1(\mathcal{H}).
\end{split}
\end{equation*}
On the other hand, in the right hand side of \eqref{fappthm},
\begin{equation*}
\begin{split}
& \hspace{0cm} \int_{H_1^{0(N)}}^{H_1^{(N)}} 
P_N^0\left[p_1\left(x,H_2^{0(N)}\right) - p_1\left(x,H_2^{(N)}\right)\right]P_N dx \\
& \hspace{0cm} = \sum \limits_{2\leq i+j\leq n~;~i,j\in \mathbb{N}} \frac{c(i,j)}{i+1}
P_N^0\left[\left(H_2^{0(N)}\right)^j-\left(H_2^{(N)}\right)^j\right]P_N
\left[\left(H_1^{(N)}\right)^{i+1} - \left(H_1^{0(N)}\right)^{i+1}\right]
\end{split}
\end{equation*}
and 
\begin{equation*}
\begin{split}
& \hspace{0cm} = \int_{H_2^{0(N)}}^{H_2^{(N)}} P_N
\left[p_2\left(H_1^{(N)},y\right)- p_2\left(H_1^{0(N)},y\right)\right]P_N^0dy\\
& \hspace{0cm} = \sum \limits_{2\leq r+s\leq m~;~r,s\in \mathbb{N}} \frac{d(r,s)}{s+1}
P_N\left[\left(H_1^{(N)} \right)^r-\left( H_1^{0(N)}\right)^r\right]P_N^0
\left[\left(H_2^{(N)}\right)^{s+1} - \left(H_2^{0(N)}\right)^{s+1}\right].
\end{split} 
\end{equation*}
In order to establish the first equality in \eqref{fappthm}, it is enough
to note that by Lemma \ref{plmmaappthm} $(iv)$ for $i,j\in \mathbb{N}$,
\begin{equation}\label{lastap}
\begin{split}
&\hspace{0cm} \textup{Tr}\{\left[(H_2^0)^j-(H_2)^j\right]
 \left[(H_1)^{i}-(H_1^0)^{i}\right]\}\\
 & = \lim_{N\longrightarrow \infty} \textup{Tr}\{P_N^0
\left[(H_2^0)^j-(H_2)^j\right]P_N
\left[(H_1)^{i}-(H_1^0)^{i}\right]P_N^0\}\\
& = \lim_{N\longrightarrow \infty} \textup{Tr}\{P_N^0
\left[\left(H_2^{0(N)}\right)^j-\left(H_2^{(N)}\right)^j\right]P_N
\left[\left(H_1^{(N)}\right)^{i} 
- \left(H_1^{0(N)}\right)^{i}\right]P_N^0\}.
\end{split}
\end{equation}
The last equality in \eqref{fappthm} follows
immediately by applying Theorem \ref{mainsec2} for the operator tuples 
$\left(H_1^{0(N)},H_2^{0(N)}\right)$ and $\left(H_1^{(N)},H_2^{(N)}\right)$
in the reducing subspaces $P_N^0\mathcal{H}$ and $P_N\mathcal{H}$ respectively.
This completes the proof.~~~~~~~~~~~~~~~~~~~~~~~~~~~~~~~$\square$

\section{Trace Formula for two variables in Infinite Dimension}\label{sec: fourth}
In this section we prove our main result namely Stokes-like formula under trace for
pairs of commuting bounded self-adjoint operators.

\begin{thm}\label{mainthm}
Let $\underline{H}^0 = \left(H_1^0,H_2^0\right)$ and 
$\underline{H} = \left(H_1,H_2\right)$ be two commuting tuples of bounded 
self-adjoint operators in a separable Hilbert
space $\mathcal{H}$ such that $H_j-H_j^0 \equiv V_j\in \mathcal{B}_2(\mathcal{H})$ for $j=1,2$.
Then there exists a unique  complex Borel measure $\mu$ on $ [a,b]^2$ such that
\begin{equation*}
\begin{split}
\textup{Tr} \{\int_{H_1^0}^{H_1} p_1(x,H_2^0) dx + \int_{H_2^0}^{H_2} p_2(H_1,y)dy + \int_{H_1}^{H_1^0} p_1(x,H_2)dx
+ \int_{H_2}^{H_2^0} p_2(H_1^0,y)dy\}\\
& \hspace{-13cm} = \textup{Tr} \{\int_{H_1^0}^{H_1} \left[p_1(x,H_2^0) - p_1(x,H_2)\right]dx
+ \int_{H_2^0}^{H_2} \left[p_2(H_1,y) - p_2(H_1^0,y)\right]dy\}\\
& \hspace{-13cm} = \int_{[a,b]^2} \left[\frac{\partial p_2}{\partial x}(x,y) - 
\frac{\partial p_1}{\partial y}(x,y)\right] \mu(dx\times dy),
\end{split}
\end{equation*}
where $p_1$ and $p_2$ are two polynomials in $[a,b]^2$ i.e. 
$p_1(x,y)$ $ = \sum \limits_{0\leq i+j\leq n} c(i,j) x^iy^j$ and
$p_2(x,y)$ $ = \sum \limits_{0\leq r+s\leq m} d(r,s) x^ry^s$ with complex coefficients and 
 $\bigcup \limits_{j=1}^2\{\sigma(H_j) \bigcup \sigma(H_j^0)\} \subseteq [a,b].$
\end{thm}
\begin{proof}
By Theorem \ref{appth1} corresponding to the tuples $\underline{H}^0$, $\underline{H}$, we conclude that
\begin{equation}\label{eqreduceq}
\begin{split}
\hspace{-0.5cm} \textup{Tr} \{\int_{H_1^0}^{H_1} p_1(x,H_2^0) dx + \int_{H_2^0}^{H_2} p_2(H_1,y)dy \\
&\hspace{-3cm} + \int_{H_1}^{H_1^0} p_1(x,H_2)dx
+ \int_{H_2}^{H_2^0} p_2(H_1^0,y)dy\}\\
& \hspace{-6.5cm} = \lim_{N\longrightarrow \infty}
\int\limits_a^b \int\limits_a^b \left[\frac{\partial p_2}{\partial x}(x,y) -
\frac{\partial p_1}{\partial y}(x,y)\right]\xi_N(x,y) dx dy,
\end{split}
\end{equation}
where $\xi_N(x,y) = \textup{Tr}\{P_N\left[E_{H_1^{(N)}}(x)  - E_{H_1^{0(N)}}(x)\right]
P_N^0\left[E_{H_2^{(N)}}(y) - E_{H_2^{0(N)}}(y)\right]P_N\}$.

For a Borel subset $\Delta$ of $[a,b]^2$, set
$$ \mu_N(\Delta) = \int \limits_{\Delta} \xi_N(x,y) dx dy$$
and observe that $\|\xi_N\|_{\infty}\leq 4 \|P_N\|_2\|P_N^0\|_2$ and therefore
$\mu_N$ is a complex Borel measure on $[a,b]^2$.
Next we want to show that there exists a complex Borel measure $\mu$ on $[a,b]^2$ such that 
for a suitable subsequence $\{N_k\}$, $\mu_{N_k}$ converges
weakly to $\mu$ i.e.
$$\lim_{k\longrightarrow \infty} \int_{[a,b]\times [a,b]} \psi(x,y) \mu_{N_k}(dx\times dy) = 
\int_{[a,b]\times [a,b]} \psi(x,y) \mu(dx\times dy)$$
for all $\psi(x,y) \in C([a,b]^2)$.
Let 
$\psi(x,y) \in C([a,b]^2)$ and let $\phi_j$($j=1,2$) be given as in \eqref{phipsieq}.
Then by applying Theorem \ref{dividthm}, for the pairs 
$\left(H_1^{(N)},H_2^{(N)}\right)$ and $\left(H_1^{0(N)},H_2^{0(N)}\right)$, we have that

\begin{equation}\label{lf1}
\begin{split}
\mathcal{I}_N = \textup{Tr}\{\int_{H_1^{0(N)}}^{H_1^{(N)}} P_N^0\phi_1\left(x,H_2^{0(N)}\right)P_N dx + 
\int_{H_2^{0(N)}}^{H_2^{(N)}} P_N^0\phi_2\left(H_1^{(N)},y\right)P_Ndy \\
& \hspace{-9cm} + \int_{H_1^{(N)}}^{H_1^{0(N)}} P_N^0 \phi_1\left(x,H_2^{(N)}\right)P_N dx
+ \int_{H_2^{(N)}}^{H_2^{0(N)}}P_N^0 \phi_2\left(H_1^{0(N)},y\right)P_Ndy\} \\
& \hspace{-12.5cm} =  \int \limits_{[a,b]^2} \int \limits_{[a,b]^2} \frac{\int \limits_{x_2}^{x_1}
\int \limits_{y_2}^{y_1} 
\psi(x,y)dx dy}{(x_1-x_2)(y_1-y_2)}\\
& \hspace{-11.5cm} \left\langle P_N^0\left(H_1^{(N)}-H_1^{0(N)}\right)P_N, 
P_N^0E_{H_1^{0(N)}}(dx_2) E_{H_2^{0(N)}}(dy_1)
 \left(H_2^{(N)}-H_2^{0(N)}\right) E_{H_1^{(N)}}(dx_1) E_{H_2^{(N)}}(dy_2)P_N\right\rangle_2,
\end{split}
\end{equation}
Next we recall from Lemma \ref{plmmaappthm} $(v)$ that 
$\sup\limits_N \left\|P_N\left(H_j^{(N)}-H_j^{0(N)}\right)P_N^0\right\|_2< C_j<\infty$
for $j=1,2$. Thus by the property of a spectral measure in a Hilbert space, one has that
\begin{equation*}
\begin{split}
&\left\|\left\langle P_N^0\left(H_1^{(N)}-H_1^{0(N)}\right)P_N, 
P_N^0E_{H_1^{0(N)}}(dx_2) E_{H_2^{0(N)}}(dy_1)
 \left(H_2^{(N)}-H_2^{0(N)}\right) E_{H_1^{(N)}}(dx_1) 
 E_{H_2^{(N)}}(dy_2)P_N\right\rangle_2  \right\|_{var}\\
 & \leq \left\|P_N^0\left(H_1^{(N)}-H_1^{0(N)}\right)P_N\right\|_2
 \left\| P_N^0\left(H_2^{(N)}-H_2^{0(N)}\right)P_N\right\|_2
 < C_1C_2.
\end{split}
\end{equation*}
On the other hand, by Theorem \ref{mainsec2}, we conclude that
\begin{equation}\label{lf2}
\begin{split}
\mathcal{I}_N = \int \limits_a^b \int \limits_a^b \left[\frac{\partial \phi_2}{\partial x}(x,y) -
\frac{\partial \phi_1}{\partial y}(x,y)\right]\xi_N(x,y) dx dy
= \int \limits_{[a,b]^2} \psi(x,y) \mu_N(dx\times dy),
\end{split}
\end{equation}
leading to the equality:
\begin{equation*}
\begin{split}
\int \limits_{[a,b]^2} \psi(x,y) \mu_N(dx\times dy) \\
& \hspace{-4.3cm} = \int \limits_{[a,b]^2} \int \limits_{[a,b]^2} \frac{\int \limits_{x_2}^{x_1}
\int \limits_{y_2}^{y_1} 
\psi(x,y)dx dy}{(x_1-x_2)(y_1-y_2)}\\
& \hspace{-4cm} \left\langle P_N^0\left(H_1^{(N)}-H_1^{0(N)}\right)P_N, 
P_N^0E_{H_1^{0(N)}}(dx_2) E_{H_2^{0(N)}}(dy_1)
 \left(H_2^{(N)}-H_2^{0(N)}\right) E_{H_1^{(N)}}(dx_1) E_{H_2^{(N)}}(dy_2)P_N\right\rangle_2,
\end{split}
\end{equation*}
for all $\psi(x,y) \in C([a,b]^2).$ Thus
\begin{equation}\label{absphi}
 \left|\int \limits_{[a,b]^2} \psi(x,y) \mu_N(dx\times dy)\right|
 < C_1C_2 \|\psi\|_{\infty}.
\end{equation}
Since, by the Riesz's theorem (page 251, \cite{royden}),
$C([a,b]^2)$ is separable in sup-norm, 
one can apply Helley's theorem (page 171, \cite{royden}) to conclude that 
there exists a subsequence $\mu_{_{N_k}}$ of
$\mu_{_{N}}$ such that $\mu_{_{N_k}}$ converges weakly to a unique complex Borel measure $\mu$ on $[a,b]^2$. i.e.
\begin{equation}\label{mscon}
 \lim \limits_{k\longrightarrow \infty} \int \limits_{[a,b]^2} \psi(x,y) \mu_{_{N_k}}(dx\times dy) =
 \int \limits_{[a,b]^2} \psi(x,y) 
 \mu(dx\times dy)~\forall ~\psi \in C([a,b]^2).
\end{equation}
This completes the proof, by applying this conclusion to the right hand side of the equation 
\eqref{fappthm}.

\end{proof}

\noindent\textbf{Acknowledgment:}  The second author acknowledges a conversation (nearly two decades old) with Voiculescu, which 
inspired this work, the primary aim of which is to obtain a two variable spectral shift 
function $\xi$ (like $\xi_N$ in Theorem \ref{mainsec2} for finite dimensions)  instead of the associated 
measure $\mu$ in Theorem \ref{mainthm}.

The first author is
grateful to Jawaharlal Nehru Centre for Advanced Scientific Research
and Indian Statistical Institute, Bangalore Centre for warm
hospitality and also thanks Department of Atomic Energy,
India through N.B.H.M Post Doctoral Fellowship for financial support.
The second author thanks Jawaharlal Nehru Centre for Advanced Scientific Research for support.

\end{document}